\documentclass[12pt]{amsart}
\usepackage{amsmath}
\usepackage{graphicx}
\usepackage{epstopdf}
\usepackage{amssymb,amscd,array}
\usepackage{color}
\usepackage{float}

\makeindex \makeatletter
\def\captionof#1#2{{\def\@captype{#1}#2}}
\makeatother

\newcounter{tablegroup}
\newcounter{subtable}[tablegroup]

\newtheorem{thm}{Theorem}[section]
\newtheorem{cor}[thm]{Corollary}
\newtheorem{lem}[thm]{Lemma}
\newtheorem{prop}[thm]{Proposition}
\newtheorem{defn}[thm]{Definition}
\newtheorem{rem}[thm]{\bf Remark}
\newtheorem{exe}[thm]{\bf Example}
\numberwithin{equation}{section}

\newcommand{\eps}{\varepsilon}

\begin{document}
\title[Dynamics of homeomorphisms of regular curves]
{Dynamics of homeomorphisms of regular curves}

\author{Issam Naghmouchi}

\address{ Issam Naghmouchi, University of Carthage, Faculty
of Sciences of Bizerte, Department of Mathematics,
Jarzouna, 7021, Tunisia.}
 \email{issam.naghmouchi@fsb.rnu.tn and issam.nagh@gmail.com}



\begin{abstract} In this paper, we prove first that the space of minimal sets of any homeomorphism $f:X\to X$ of a regular curve $X$ is closed in the hyperspace $2^X$ of closed subsets of $X$ endowed with the Hausdorff metric, and the non-wandering set $\Omega(f)$ is equal to the set of recurrent points of $f$. Second, we study the continuity of the map $\omega_f:X\to 2^X;x\mapsto \omega_f(x)$, we show for instance the equivalence between the continuity of $\omega_f$ and the equality between the $\omega$-limit set and the $\alpha$-limit set of every point in $X$. Finally, we prove that there is only one (infinite) minimal set when there is no periodic point.
\end{abstract}
\maketitle

\section{\bf Introduction}
A continuum is a nonempty connected metric compact space. A Peano continuum is a locally connected continuum. It is well known that Peano continua are arcwise connected and locally arcwise connected (see Theorems 8.23 and 8.25 in Nadler book \cite{Nadler}). A continuous image of a Peano continuum is a Peano continuum
( Proposition 8.16, \cite{Nadler}). A continuum $X$ is a \textit{regular curve}  if for each $x\in X$ and
each open neighborhood $V$ of $x$ in $X$, there exists an open neighborhood $U$ of $x$ in $X$
such that $U\subset V$ and the boundary set $\partial(U)$ of $U$ is finite. Each
regular curve is a Peano $1$-dimensional continuum. It follows that each regular curve is locally arcwise connected. Note that any graph as well as any local dendrite is a regular curve (for definitions and more details see \cite{Kur} and \cite{Nadler}).

 Let $X$ be a compact metric space, the closure (respectively the boundary set) of a subset $A$ of $X$ is denoted by $\overline{A}$ (respectively $\partial A$). An open neighborhood of a non-empty subset $A$ is an open set of $X$ containing $A$. An $\eps$-neighborhood of $A$ is an open neighborhood of $A$ included into $\cup_{x\in A} B(x,\eps)$ where $B(x,\eps)$ denotes the open ball of center $x$ and radius $\eps$.  Let $f:X\to X$ be a continuous map. The \textit{forward orbit} (under $f$) of a given point $x\in X$ is the set $O^{+}_f(x):=\{f^n(x): \ n\in\mathbb{Z}_{+}\}$. A point $x\in X$ is said to be \textit{periodic for $f$} if for some $n\in\mathbb{N}$, $f^n(x)=x$. The orbit of periodic point is called a \textit{periodic orbit}. We define the
\textit{$\omega$-limit set} of a point $x$ with respect to $f$ to be
the set $\omega_{f}(x) := \{y\in X: \exists\ n_{i}\in \mathbb{N},
n_{i}\rightarrow\infty, \lim_{i\rightarrow\infty}d(f^{n_{i}}(x), y)
= 0\}$. An $\omega$-limit set $\omega_{f}(x)$ is always a non-empty, closed and strongly
invariant set, i.e. $f(\omega_{f}(x))= \omega_{f}(x)$ (Lemma 2, Chapter IV in \cite{blo}) and has the following property known as ''weak incompressibility'': $F \cap f(\omega_{f}(x) \ F) \neq \emptyset $ whenever $F$ is a proper, nonempty and closed subset of $\omega_{f}(x)$ (Lemma 3, Chapter IV in \cite{blo}). Moreover, an $\omega$-limit set $\omega_{f}(x)$ is finite if and only if $x$ is asymptotic to a periodic point (Lemma 4, Chapter IV in \cite{blo}). A subset $M$ of $X$ is called \textit{minimal} (for $f$) whenever it is non-empty, closed, strongly invariant while no proper subset of $M$ has these properties. A point $x$ is said to be recurrent for $f$ if $x\in\omega_f(x)$. A point $x\in X$ is called \textit{wandering} for $f$ if there exists some neighbourhood $U$ of $x$ such that $f^{-n}(U)\cap U=\emptyset$ for every $n\in \mathbb{N}$. Otherwise, the point $x$ is said to be \textit{non-wandering}. Denote by $\Omega(f)$ (resp. $R(f)$) the set of non-wandering point of $f$ (resp. the set of recurrent points of $f$).

If $f:X\to X$ is a homeomorphism then the \textit{full orbit} (under $f$) of a given point $x\in X$ is the set $O_f(x):=\{f^n(x): \ n\in\mathbb{Z}\}$ and the \textit{backward orbit}) of $x$ is the set $O^{-}_f(x):=\{f^n(x): \ n\in\mathbb{Z}_{-}\}$). The \textit{$\alpha$-limit set} of a point $x$ is defined as the $\omega$-limit set of $x$ with respect to $f^{-1}$.  Notice that a subset $M$ is minimal for $f$ if and only if it is minimal for $f^{-1}$.

In the hyperspace $2^X$ of all closed subsets of $X$ endowed with the Hausdorff metric, the space of all $\omega$-limit sets is not always closed for an arbitrary continuous map $f:X\to X$ (see \cite{gui} for examples on the square $[0,1]^2$ and \cite{kocan} for example on dendrite). However the compactness of this space has been established for interval maps (see \cite{Blokh}) and later for graph maps (see \cite{mai}). In this paper, we prove the compactness of this space for regular curves's homeomorphisms. Lets recall some results obtained in \cite{Nagh1} for this kind of maps which will be useful in the sequel:

\begin{thm}\cite{Nagh1}\label{mainthm}
If $f$ is a homeomorphism of a regular curve then the following hold:\\
(i) Any $\omega$-limit set (respectively $\alpha$-limit set) is minimal.\\
(ii) For any $x\in X$, $\omega_f(x)=\alpha_f(x)$ provided that $\omega_f(x)\cup\alpha_f(x)$ is infinite.
\end{thm}

This paper is organized as follow: In section 2, we prove first that for any homeomorphism of a regular curve $f$, $\Omega(f)=R(f)$, then we prove the compactness in the Hausdorff sense of the space of all $\omega$-limit sets. Then in Section 3, we study the map $\omega_f:X\to 2^X; x\mapsto\omega_f(x)$, we show the continuity of this map at every point with infinite $\omega$-limit set and then we establish the equivalence between the continuity of $\omega_f$ and the equality between the $\omega$-limit set and the $\alpha$-limit set of every point in $X$. Moreover, as a result of the equicontinuity of  $f$ on $X\setminus \Omega(f)$, we deduce the continuity of $\omega_f$ on $X\setminus \Omega(f)$. Finally in Section 4, we prove that there is only one minimal set when the set of periodic points is empty.

\section{\bf The non-wandering set and the space of minimal sets with respect to the Hausdorff metric}
\begin{thm}\label{wand} If $f:X\to X$ is a homeomorphism of a regular curve $X$ then $\Omega(f)=R(f)$.
\end{thm}
\begin{proof}
Let $x\in \Omega(f)$ then there is a sequence $(x_n)_n$ of points in $X$ and an infinite sequence of positif integers $(k_n)_n$ such that $\lim_{n\to +\infty} x_n=\lim_{n\to +\infty} f^{k_n}(x_n)=x$ (see Theorem 5.7 in \cite{Walters}). Assume that $x$ is not recurrent then for some $\delta>0$, $d(f^n(x),x)>\delta$ for any $n\in\mathbb{N}$. As $X$ is a regular curve there is an open neighborhood  $U\subset B(x,\delta)$ of $x$ with finite boundary. Fix an $\eps>0$, then because $X$ is locally arcwise connected, for $n$ large enough there is an arc $I_n$ included into $B(x,\eps)\cap U$ joining $x$ and $x_n$, thus $f^{k_n}(I_n)\cap \partial U\neq \emptyset$. As the boundary of $U$ is finite, there is a point $b_{\eps}\in \partial U$ such that $\alpha_f(b)\cap \overline{B(x,\eps)}\neq\emptyset$. Again as the boundary of $U$ is finite and $\eps$ is chosen arbitrary, there is $b\in\partial U$ such that $x\in\alpha_f(b)$. By Theorem \ref{mainthm}, $\alpha_f(b)$ is minimal thus $x\in R(f)$.
\end{proof}


 \medskip

Given a compact metric space $X$ with a metric $d$, we denote by $2^X$ the hyperspace of all nonempty closed subsets of $X$. For any two subsets $A$ and $B$ of $X$, we denote by $d(A,B)=inf_{x\in A,y\in B}d(x,y)$ and $d(x,A)=d(\{x\},A)$. The Hausdorff metric $d_H$ on $2^X$ is defined as follows : Let $A,B\in 2^X$
$$d_H(A,B)=max\{sup_{x\in A}d(x,B),sup_{y\in B}d(y,A)\}.$$\\
This defines a distance on $2^X$ (\cite{Nadler}, Theorem 4.2). With this distance, $2^X$ is a compact metric space (\cite{Nadler}, Theorem 4.3).

\begin{thm}\label{stab} Let $f:X\to X$ be a homeomorphism of a regular curve $X$ and let $(M_n)_n$ be a sequence of minimal sets that converges with respect to the Hausdorff metric to $M\in 2^X$ then $M$ is minimal.
\end{thm}

\begin{lem}\label{stab1} Under the assumptions of Theorem \ref{stab}, if $M\subset P(f)$ then $M$ is a periodic orbit.
\end{lem}
\begin{proof}
If $M$ is countable and infinite then there is $a\in M$ such that $O_f(a)=\{a,\dots, f^{p-1}(a)\}$ is open in $M$. By Proposition 3.1 in \cite{dan}, $M$ satisfies the weak incompressibility property, thus if we consider the clopen set $F=O_f(a)$ of $M$, there exists $b\notin O_f(a)$ such that $f(b)\in O_f(a)$, contradiction with the fact that $b$ itself is periodic.

If $M$ is uncountable then for some $p\in \mathbb{N}$, $M\cap Fix(f^p)$ is uncountable. Hence, there exists a non-empty open set $U$ of $X$ satisfying the following properties:

(i) There exists $z\in  M\setminus \left(\overline{\cup_{0\leq i\leq p-1} f^i(U)}\right)$,

(ii) $U\cap Fix(f^p)\cap M$ is infinite,

(iii)  $\partial U$ is finite.

Take any point $a\in U\cap Fix(f^p)\cap M$. Let $\eps>0$ be such that $B(z,\eps)\subset X\setminus \left(\overline{\cup_{0\leq i\leq p-1} f^i(U)}\right)$. As $X$ is a regular curve, for any $\mu\in (0,\eps]$, there exists an arcwise connected neighborhood $U_{\mu}$ of $a$ in $X$ such that $diam(U_{\mu})<\mu$ and $U_{\mu}\subset U$. Since $\lim_{n\to +\infty}d_H(M_n,M)=0$, there is $n\in\mathbb{N}$ such that $$M_n\cap U_\eps\neq\emptyset\neq M_n\cap \left(X\setminus \overline{\cup_{0\leq i\leq p-1} f^i(U)}\right).$$
Take a point $x\in M_n\cap U_{\eps}$, then by minimality of $M_n$, $f^{i}(x)\in X\setminus \overline{\cup_{0\leq i\leq p-1} f^i(U)}$ for infinitely many $i\in\mathbb{N}$. It turns out that there is $r\in\{0,\dots,p-1\}$ such that for infinitely many $i\in\mathbb{N}$, $f^{ip+r}(U_{\eps})\cap \partial (f^{r}(U))\neq\emptyset$. So there is $b\in f^{ip}(U_{\eps})\cap \partial (U)$ for infinitely many $i\in\mathbb{N}$. Therefore the negative orbit of $b$ visits $U_{\eps}$ infinitely many times which implies that $\alpha_f(b_j)\cap \overline{U_{\eps}}\neq\emptyset$. We conclude that for any $0<\mu\leq\eps$, there is a point from the boundary of $U$ which has an $\alpha$-limit set with non-empty intersection with $B(a,\mu)$. It follows from property (iii) that for some fixed $b\in \partial(U)$, $a\in \alpha_f(b)$.
We proved that any $a\in U\cap Fix(f^p)\cap M$ belongs to the $\alpha$-limit set of a point from the boundary of $U$. Again as the boundary set $\partial (U)$ is finite, there a point $b\in\partial (U)$ such that $\alpha_f(b)$ contains at least two periodic points $a$ and $a^{'}$  with disjoint orbits which induices a contradiction with the minimality of $\alpha_f(b)$. Consequently, $M$ must be finite. Again as $M$ satisfies the weak incompressibility property (Proposition 3.1, \cite{dan}), it must be a periodic orbit.
\end{proof}

\bigskip

Similarly, we prove the following Lemma:
\begin{lem}\label{stab2}
Under the assumptions of Theorem \ref{stab}, any point in $M$ with infinite orbit belongs to a minimal subset of $M$.
\end{lem}
\begin{proof} Assume the contrary. There is $z\in M$ such that $z\notin \omega_f(z)$, denote by $N=\omega_f(z)$, it is a minimal subset of $M$. Let $U$ be a closed neighborhood of $z$ with finite boundary and disjoint from $N$. Let $\eps>0$ and $U_{\eps}$ be an arbitrary arcwise connected neighborhood of $z$ included into $U$ and with diameter less than $\eps$. As $\lim_{n\to +\infty}d_H(M_n,M)=0$, there is $n\in\mathbb{N}$ such that $M_n\cap U_{\eps}\neq\emptyset\neq M_n\cap X\setminus U$. Take $x\in M_n\cap U_{\eps}$ then for infinitely many $i\in\mathbb{N}$, $f^i(x)\in U_{\eps}$ and for $i\in\mathbb{N}$ large enough, $f^i(z)\notin U$. So $f^{i}(U_{\eps})$ intersects the boundary of $U$ infinitely many times. In result, some point in the boundary of $U$ has an $\alpha$-limit set with non-empty intersection with $\overline{U_{\eps}}$. As $\eps$ is chosen arbitrarily and the boundary of $U$ is finite, we can conclude that $z$ belongs to the $\alpha$-limit set of some point from the boundary of $U$, a contradiction.
\end{proof}

\medskip

\textbf{\textit{Proof of Theorem \ref{stab}.}} Assume that $M$ is not minimal then by Lemmas \ref{stab1} and \ref{stab2}, there are two disjoint minimal subsets $N_1$ and $N_2$ of $M$ such that $N_1$ is infinite. Let $U$ be an open neighborhood of $N_1$ with finite boundary such that $N_2\cap \overline{U}=\emptyset$, set $k=Card(\partial U)$. Let $V$ be an open neighborhood of $N_2$ such that $\left(\cup_{0\leq i\leq k} f^i(V)\right)\cap \overline{U}=\emptyset$. Take a point $z\in N_1$, then $z$ has an infinite orbit since $N_1$ is infinite. So let $\delta>0$ be such that the balls $B(f^i(z),\delta), \ i=0,\dots,k$ are pairwise disjoint. As $f$ is uniformly continuous and $X$ is a regular curve, there is an arcwise connected neighborhood $U_{\delta}$ of $z$ such that $f^i(U_{\delta})\subset B(f^i(z),\delta)$ for $i=0,\dots,k$. Since $\lim_{n\to +\infty}d_H(M_n,M)=0$, there is $n\in\mathbb{N}$ such that $M_n\cap U_{\delta}\neq\emptyset\neq M_n\cap V$ so take a point $x\in M_n\cap U_{\delta}$ then by the minimality of $M_n$, $f^j(x)\in V$ for some $j\in\mathbb{N}$. Hence, $f^{j+i}(U_{\delta})$ intersects the boundary of $U$ for $i=0,\dots,k$ which implies that the boundary of $U$ has at least $k+1$ points, a contradiction. In conclusion, $M$ is minimal. $\qed$

\medskip

\section{\bf On the map $x\mapsto \omega_f(x)$}

\begin{thm}\label{continf} Let $f:X\to X$ be a homeomorphism of a regular curve $X$, then the map $\omega_f$ is continuous at any point with infinite $\omega$-limit set.
\end{thm}

To prove Theorem \ref{continf}, we need the following Lemma:
\begin{lem}\label{nmin}\cite{Nagh1} If $M$ is an infinite minimal set and $U$ is a neighborhood of $M$ with finite boundary where $k=Card(\partial U)$ then there is an open neighborhood $V\subset U$ of $M$ such that for any $x\in V$ and for any $n\in \mathbb{Z}$, $\{f^{n}(x),\dots,f^{n+k}(x)\}\cap U\neq\emptyset$.
\end{lem}
\medskip

\textit{Proof of Theorem \ref{continf}.}
Let $(x_n)_n$ be a sequence in $X$ that converges to a point $x\in X$ with infinite $\omega$-limit set. Let $\eps>0$ and let $U$ be an $\eps$-neighborhood of $\omega_f(x)$ with finite boundary, set $k=Card(\partial U)$. Let $V$ be as in Lemma \ref{nmin},  then there is $N>0$ such that $f^N(x)\in V$ so for $n$ large enough, $f^N(x_n)\in V$. Following Lemma \ref{nmin}, $\omega_f(x_n)\cap \overline{U}\neq\emptyset$ for $n$ large enough. Therefore any limit of any convergent subsequence of $(\omega_f(x_n))_n$ (with respect to the Hausdorff metric) has a non-empty intersection with $\omega_f(x)$. By Theorem \ref{stab}, the Hausdorff limit of any convergent subsequence of $(\omega_f(x_n))_n$ is minimal, hence it is equal to $\omega_f(x)$ (since $\omega_f(x)$ itself is minimal, see Theorem \ref{mainthm}). Consequently, $(\omega_f(x_n))_n$ converges in $2^X$ to $\omega_f(x)$. We conclude then the continuity of $\omega_f$ at $x$. $\qed$

\bigskip
Let $X$ be a compact metric space and $f:X\to X$ be a continuous map. A sequence $\langle x_{-n}\rangle_{n=0}^{\infty}$ of points in $X$ is called a negative orbit or negative trajectory through $x$ if $x_0= x$ and $f(x_{-n-1})=x_{-n}$ for every integer $n\geq 0$.

\begin{defn}
Let $X$ be a compact metric space and $f:X\to X$ be a continuous map then the $\alpha$-limit set of a negative orbit $\langle x_{-n}\rangle_{n=0}^{\infty}$ is the set $\alpha(\langle x_{-n}\rangle_{n=0}^{\infty},f)$ of all limit points of the sequence $\langle x_{-n}\rangle_{n=0}^{\infty}$.
\end{defn}

Notice that in the case of a homeomorphism, for any $x\in X$, there is only one negative orbit with starting point $x$ and so the $\alpha$-limit set of the unique negative orbit starting from $x$ coincide with the $\alpha$-limit set of $x$ with respect to $f$ which is in fact $\omega_{f^{-1}}(x)$.

\begin{lem}\label{op}\cite{oprocha}
 For any compact space $(X, d)$, any continuous map $f:X\to X$ and any negative trajectory $ \langle x_{-n}\rangle_{n=0}^{\infty}$, the set $\alpha(\langle x_{-n}\rangle_{n=0}^{\infty},f)$ is nonempty, closed and strongly invariant
\end{lem}

\begin{thm}\label{nega} Let $X$ be a compact metric space and $f:X\to X$ be a continuous map. If the map $\omega_f$ is continuous then the following two assertions hold:\\
(i) For each $x\in X$, $\omega_f(x)$ is minimal.\\
(ii) For  each $x\in X$ and for any negative trajectory $ <x_{−n}>_{n=0}^{\infty}$ with starting point $x_0=x$, $\omega_f(x)$ is the unique minimal set included into $\alpha(\langle x_{-n}\rangle_{n=0}^{\infty},f)$.
\end{thm}
\begin{proof}
 (i) Let $x\in X$ and let $M$ be a minimal subset for $f$ included into $\omega_f(x)$. Take $a\in M$ then for some infinite sequence of positive integer $(n_i)_i$, $\lim_{i\to +\infty} f^{n_i}(x)=a$ and for each $i\in\mathbb{N}$, $\omega_f(f^{n_i}(x))=\omega_f(x)$. From the continuity of $\omega_f$, we get $\omega_f(a)=\omega_f(x)$. In result, $\omega_f(x)$ is minimal.\\
(ii) Let $ <x_{−n}>_{n=0}^{\infty}$ be a negative trajectory with starting point $x_0=x$. By Lemma \ref{op}, there is a minimal set $M$  for $f$ included into $\alpha(\langle x_{-n}\rangle_{n=0}^{\infty},f)$. Let $a\in M$ and let $(x_{-n_i})_i$ be a subsequence of $(x_{-n})_n$ that converges to $a$. Then clearly, $\omega_f(x_{-n_i})=\omega_f(x)$ for each $i\in\mathbb{N}$ and $\omega_f(a)=M$. By continuity of the map $\omega_f$, $M=\omega_f(x)$.
\end{proof}

\begin{rem}
  \rm{
  (1) In the concluded assertion (ii) of Theorem \ref{nega}, we cannot hope more than the inclusion $\omega_f(x)\subset\alpha(\langle x_{-n}\rangle_{n=0}^{\infty},f)$, a strict inclusion was illustrated in example 1 for the homeomorphism $T$.\\
  (2) Also note that both assertions (i) and (ii) of Theorem \ref{nega} are not sufficient to ensure the continuity of the map $\omega_f$, in the following we furnish a simple example of a homeomorphism $F$  where the $\omega$-limit set of any point coincide with its $\alpha$-limit set and is minimal while the map $\omega_F$ is not continuous.

  }
\end{rem}

\begin{exe}\rm{
  Both homeomorphisms $T$ and $F$ are defined on the same compact space $Y$ which is a subset of the real plane:
  $$Y=\{(0,0)\}\cup \{(\frac{1}{n},0): n\in\mathbb{Z}^{*}\}\cup\{(0,\frac{1}{n}): n\in\mathbb{N}\}\cup \bigcup_{n\in\mathbb{N}} A_n,$$
  where for each $n\in\mathbb{N}$, $A_n=\{(\frac{1}{k},\frac{1}{n}): \mid k \mid\leq n \ and \ k\neq 0\}$.

  \medskip

  (1) The map $T$ is defined as follow: First let $T(0,0)=(0,0)$ and $T(-1,0)=(1,0)$. For each $n\in\mathbb{N}$ and for each $k\in \{-n,\dots,-1,1,\dots,n\}$,\\
  $T(\frac{1}{n},0)=(\frac{1}{n+1},0)$; $T(0,\frac{1}{n})=(0,\frac{1}{n+1})$;\\
  $T(\frac{1}{-(n+1)},0)=(\frac{1}{-n},0)$;\\
  $T(-1,\frac{1}{n+1})=(1,\frac{1}{n})$; $T(-1,1)=(0,1)$\\
  $T(\frac{1}{k},\frac{1}{n})=(\frac{1}{k+1},\frac{1}{n})$  if  $k<-1$ or $1\leq k\leq n-1$; \\
  $T(\frac{1}{n},\frac{1}{n})=(\frac{1}{-n},\frac{1}{n})$.

  Obviously, $T$ is a homeomorphism of $Y$ onto itself and the map $\omega_T$ is a constant hence continuous while we have the following strict inclusion $\omega_T(0,1)=\{(0,0)\}\subsetneq \alpha_T(0,1)= \{(\frac{1}{n},0): n\in\mathbb{Z}^{*}\}\cup \{(0,0)\}$.

  \medskip

  (2) The map $F$ is defined as follow:

   First let $F(0,0)=(0,0)$ and $F(0,1)=(0,1)$ and for each $n\in\mathbb{N}$, for each $k\in \{-n,\dots,-1,1,\dots,n\}$,\\
  $F(\frac{1}{n},0)=(\frac{1}{n+1},0)$;\\
  $F(\frac{1}{-(n+1)},0)=(\frac{1}{-n},0)$;  $F(-1,0)=(1,0)$;\\
  $F(\frac{1}{k},\frac{1}{n})=(\frac{1}{k+1},\frac{1}{n})$  if  $k<-1$ or $1\leq k\leq n-1$; \\
  $F(\frac{1}{n},\frac{1}{n})=(0,\frac{1}{n})$; $F(0,\frac{1}{n})=(\frac{1}{-n},\frac{1}{n})$ and $F(-1,\frac{1}{n})=(1,\frac{1}{n})$.

  It is easy to verify that $F$ is a homeomorphism  of $Y$ onto it self. For each $n\in\mathbb{N}$ let $B_n=A_n\cup \{(0,\frac{1}{n})\}$. Thus $(B_n)_n$ is a sequence of periodic orbits that converges with respect to the Hausdorff metric to the set $A=\{(0,0)\}\cup \{(\frac{1}{n},0): n\in\mathbb{Z}^{*}\}$. For each $x\in A$, $\omega_F(x)=\alpha_F(x)=\{(0,0)\}$ and any point in $Y\setminus A$ is periodic, so both assertions (i) and (ii) of Theorem \ref{nega} hold while the map $\omega_F$ is not continuous.
  }
\end{exe}

However in the case of homeomorphisms of regular curves we get the following:
\begin{thm}\label{contomega} Let $f:X\to X$ be a homeomorphism of a regular curve $X$, then the following assertions are equivalent:

\rm{(1)} The map $\omega_f$ is continuous.

\rm{(2)} The $\omega$-limit set of any point coincide with its $\alpha$-limit set.

\rm{(3)} The map $\alpha_f$ is continuous.
\end{thm}

\begin{proof} (2) $\Rightarrow$ (1): Assume that the $\omega$-limit set of any point coincide with its $\alpha$-limit set. By Theorem \ref{continf}, we have to examine only the continuity of $\omega_f$ at points with finite $\omega$-limit set. Let $x\in X$ with finite $\omega$-limit set and suppose that $\omega_f$ is not continuous at $x$, then there is a sequence $(x_n)_n$ in $X$ that converges to $x$ such that $\lim_{n\to\infty} \omega_f(x_n)=L$ in the sense of Hausdorff and $L\cap \omega_f(x)=\emptyset$. Since $x$ has finite $\omega$-limit set, it is asymptotic to a periodic point $a\in\omega_f(x)$, hence one can find a sequence $(y_n)_n$ that converges to $a$ such that $\lim_{n\to\infty}\omega_f(y_n)=L$. Without loss of generality, one may assume that $a$ is a fixed point (since otherwise instead of $f$ we consider $g=f^p$ where $p$ is the period of $a$). Let $U$ be an open neighborhood of $a$ with finite boundary (set $k=Card (\partial U)$) such that $L\cap \overline{U}=\emptyset$. For each $\eps>0$, there is an arc $I_{\eps}$ in $U$ joining $a$ and $y_n$ for some $n\in\mathbb{N}$ and with diameter less than $\eps$ and as $f^k(y_n)$ leaves $\overline{U}$ for $k$ large enough, there is $b_{\eps}\in\partial U$ such that $\alpha_f(b_{\eps})\cap I_{\eps}\neq\emptyset$. Therefore, there is $b\in\partial U$ such that $\omega_f(b)=\alpha_f(b)=\{a\}$ and for infinitely many $n$, there is an arc $I_n$ with diameter less than $\frac{1}{n}$, joining in $U$ the points $a$ and $y_{m(n)}$ (for some $m(n)\in\mathbb{N}$) and containing at least one point $b_n$ from the backward orbit of $b$, notice that $b_n$ is distinct from both $a$ and $y_{m(n)}$. So it is possible to construct $k+1$ pairwise disjoints arcs $J_1,\dots,J_{k+1}$ in $U$ such that for every $l\in\{1,\dots,k+1\}$, one end point of $J_l$ has an $\omega$-limit set disjoint from $\overline{U}$ and the other has an  $\omega$-limit set equal to $\{a\}$ (the construction can be done in the following way: Let $J_1$ be the sub-arc of $I_1$ joining $y_{m(1)}$ and $b_1$ hence $a\notin J_1$. So for some $n\in\mathbb{N}$, $I_n$ is disjoint from $J_1$, take $J_2$ the sub-arc of $I_n$ joining $y_{m(n)}$ and $b_n$. Similarly, there is $n'\in\mathbb{N}$ such that $I_{n'}$ is disjoint from $J_1\cup J_2$, so let $J_3$ be the sub-arc of $I_{n'}$ joining $y_{m(n')}$ and $b_{n'}$, and we proceed by induction until we get $J_{k+1}$). For $m$ large enough and $l=1,\dots,k+1$, $f^m(J_l)$ has a point outside of $\overline{U}$ and another inside of $U$ thus $f^m(J_l)\cap \partial U\neq\emptyset$ which implies that $Card(\partial U)>k$, a contradiction. In result, $\omega_f$ is continuous at $x$.

(1) $\Rightarrow$ (2): Follows from Theorem \ref{nega} and [Theorem \ref{mainthm}, (i)].

Finally, by applying the equivalence between assertions (1) and (2) to $f^{-1}$, we obtain the equivalence between assertions (3) and (2).

\end{proof}

\medskip

\begin{defn}
Let $X$ be a compact metric space and $f:X\to X$ be a continuous map. A point $x$ is said to be an equicontinuous point of the dynamical system $(X,f)$ provided that for any $\eps>0$, there is $\delta>0$ such that for any $y\in B(x,\delta)$, $d(f^n(x),f^n(y))<\eps$ for all $n\in\mathbb{Z}_{+}$. The dynamical system $(X,f)$ is said to be equicontinuous provided that every point in $X$ is equicontinuous.
\end{defn}

\medskip

\begin{prop} Let $f:X\to X$ be a homeomorphism of a hereditarily locally connected continuum $X$, then any point $x\in X\setminus \Omega (f)$ is equicontinuous for $(X,f$).
\end{prop}

\begin{proof} Suppose that $x\in X\setminus \Omega (f)$ then for some connedted neighborhood $U$ of $x$ in $X$, $\{f^n(U):n\in\mathbb{N}\}$ is pairwise disjoint. As $X$ is a hereditarily locally connected continuum,  $\{f^n(U):n\in\mathbb{N}\}$ is a null family. It follows that $x$ is an equicontinuous point for $(X,f)$.
\end{proof}

\begin{cor}
For any homeomorphism $f:X\to X$ a hereditarily locally connected continuum $X$, the map $\omega_f$ is continuous on $X\setminus \Omega(f)$
\end{cor}

\begin{rem}\rm{
It is well known that the equicontinuity property of a dynamical system $(X,f)$ implies but is not implied by the continuity of $\omega_f$. Bruckner and Ceder \cite{Bru} has shown the equivalence between these two properties in the case of continuous interval map. This equivalence is not true in the case of the circle, take for example any cicle's homeomorphism $f$ that preserves a Cantor set, it has a unique minimal set so the map $\omega_f$ is constant but $(S^1,f)$ is not equicontinuous.}
\end{rem}

\section{\bf A Dynamical Criterion for the existence of periodic point}

The aim of this section is to answer the following question posed in \cite{Nagh1}: Is there a homeomorphism of a regular curve without periodic point
and having several minimal sets?

\begin{thm}\label{dyncriter} Any homeomorphism of a regular curve without periodic points possesses a unique minimal set.
\end{thm}

We will use the following notations: For a regular curve homeomorphism $f:X\to X$, we denote by\\
\begin{itemize}
  \item $\mathcal{M}(f)$ the family of all minimal sets of $f$,
  \item $\mathcal{M}_{c}(f)$ the family of minimal sets of $f$ having finite number of connected components,
  \item for each $M\in \mathcal{M}_{c}(f)$, $mesh(M):=\sup\{diam(C): C$ is a connected component of $M \}$
\end{itemize}

\begin{lem}\label{fc}
If $f:X\to X$ is a regular curve homeomorphism without periodic points and has at least two minimal sets, then $\mathcal{M}_{c}(f)$ is not dense in $\mathcal{M}(f)$ where $\mathcal{M}(f)$ is endowed with the Hausdorff metric.
 \end{lem}

 \begin{proof}
 Assume the contrary, that is there exists a regular curve homeomorphism $f:X\to X$ having the following properties:

  (i) $P(f)=\emptyset$,

  (ii) $Card(\mathcal{M}(f))\geq 2$ and

  (iii) $\overline{\mathcal{M}_{c}(f)}=\mathcal{M}(f)$.

First observe that $\lim_{n\to +\infty} mesh(M_n)=0$ whenever $(M_n)_n$ is a pairwise distinct sequence in  $\mathcal{M}_{c}(f)$ (this is due to the fact that $X$ is finitely Suslinian, see Corollary 1.7 in \cite{lel}). Let $M$ be a connected minimal set of $f$ (if there is no such minimal set, then pick any element $ M$ from $\mathcal{M}_{c}(f)$ and consider $f^l$ instead of $f$ where $l$ is the number of connected component of $M$) and fix two distinct points $a,b\in M$. Suppose $\delta=d(a,b)$ and take an arcwise connected neighborhood $U_a$ of $a$ in $X$ and a neighborhood $U_b$ of $b$ in $X$ such that any arc in $X$ joining a point in $U_a$ and a point in $U_b$ has diameter at least $\frac{\delta}{2}$. According to Theorem \ref{continf}, the map $\omega_f$ is continuous, so $\omega_f(X)=\mathcal{M}$ is a non-degenerate continuum with respect to the Hausdorff metric, in particular $\mathcal{M}_{c}(f)$ has no isolated point. Therefore, there is a pairwise distinct sequence $(M_n)_n$ in $\mathcal{M}_{c}(f)$ with Hausdorff limit $M$. There is $N_0\in\mathbb{N}$ such that for any $n\geq N_0$, some connected component of $M_{n}$ is entirely included into $U_a$. So take a connected component $C_0$ of $M_{N_0}$ included into $U_a$. Consider an arc $J_0$ joining a point $c_0\in C_0$ and a point $a'_0\in U_a$ such that $J_0\cap M=\{a'_0\}$, let $a_0\in J_0$ be such that $J_0\cap \{x\in J_0:\omega_f(x)=M\}=\{a_0\}$, this is possible since from the continuity of $\omega_f$, the set $\omega_f^{-1}\{M\}\cap J_0$ is closed in $J_0$. Denote by $\gamma_0$ the number of connected components of $M_{N_0}$. Hence $C_0$ is $f^{\gamma_0}$-invariant and by the connectedness of $M$, $M$ still minimal for $f^{\gamma_0}$. Thus for some $t_0\in\mathbb{N}$, $f^{t_0\gamma_0}(a_0)\in U_b$. Denote by $I_0$ the sub-arc of $J_0$ joining $c_0$ and $a_0$, hence $diam(f^{t_0\gamma_0}(I_0))\geq \frac{\delta}{2}$.
Take $\zeta_0\in f^{t_0\gamma_0}(I_0)\setminus\{f^{t_0\gamma_0}(a_0)\}$ such that the sub-arc $L_0$ of $f^{t_0\gamma_n}(I_0)$ joining $f^{t_0\gamma_0}(c_0)$ and $\zeta_0$ has diameter at least $\frac{\delta}{4}$. By the continuity of the map $\omega_f$, $\{\omega_f(x): x\in L_0\}$ is closed in $\mathcal{M}(f)$ and does not contain $M$. By applying the continuity of $\omega_f$ at $a'_0$, there is $\mu_0>0$ such that $\omega_f(x)\in \mathcal{M}(f)\setminus \{\omega_f(y): y\in L_0\}$ whenever $d(x,a'_0)<\mu_0$. It is possible then to find $N_1>N_0$ such that some connected component of $M_{N_1}$ say $C_1$ is a subset of $U_a$ and there is an arc $J_1$ in $U_a$ joining some point $c_1\in C_1$ and a point $a'_1\in U_1$ such that $d(y,a'_0)<\mu_0$ for every $y\in J
_1$ and $J_1\cap M=\{a'_1\}$. By a similar way we define $a_1$, $\gamma_1$, $t_1$, $I_1$, $\zeta_1$ and $L_1$. Thus $L_{1}\cap L_0=\emptyset $ and recall that we have already $diam(L_1)\geq\frac{\delta}{4}$. By same arguments as above, there is $\mu_1>0$ such that  $\omega_f(x)\in \mathcal{M}(f)\setminus \{\omega_f(y): y\in L_0\cup L_1\}$ whenever $d(x,a'_1)<\mu_1$. One could find then $N_2>N_1$ such that some connected component $C_2$ of $M_{N_2}$ is a subset of $U_a$ and there is an arc $J_2$ in $U_a$ joining some point $c_2\in C_2$ and a point $a'_2\in M$ such tat $d(y,a'_1)<\mu_1$ for every $y\in J_2$ and $J_2\cap M=\{a'_2\}$. Similary, we define $a_2$, $\gamma_2$, $t_2$, $I_2$, $\zeta_2$ and $L_2$. Thus $(L_1\cup L_0)\cap L_2=\emptyset $ and $diam(L_2)\geq\frac{\delta}{4}$. We continue this process inductively infinitely many times which lead to the construction of a pairwise disjoint sequence of arcs $(L_n)_n$ with diameter at least $\frac{\delta}{4}$, a contradiction. We conclude then that $\mathcal{M}_c(f)$ is not dense in $\mathcal{M}(f)$.

 \end{proof}

\textbf{\textit{Proof of Theorem \ref{dyncriter}.}}  Let $f:X\to X$ be a homeomorphism of a regular curve $X$ without periodic points. According to Theorems \ref{contomega} and \ref{mainthm}, the map $\omega_f$ is continuous, so $\omega_f(X)=\mathcal{M}(f)$ is a locally connected continuum with respect to the Hausdorff metric, in particular it is locally arcwise connected. Assume that $\mathcal{M}(f)$ is not reduced to a single point. By Lemma \ref{fc}, there is a non empty open subset $O$ in $\mathcal{M}(f)$ such that any element in $O$ has infinitely many connected components. Because of the local connectedness property of $\mathcal{M}(f)$, one could find an arc $A_{[0,1]}=\{A_t,t\in[0,1]\}$ in $O$ and a homeomorphism $h:[0,1]\to A_{[0,1]};t\mapsto A_t$. For each $n\in\mathbb{N}$, let $a^n(0)=0<\dots<1=a^n({l_n})$ be a subdivision of $[0,1]$ with step less than $\frac{1}{n}$.

For each $n\in\mathbb{N}$, one can construct by an induction process a set $C_n=\{x^n(1),\dots,x^n({l_n})\}$ satisfying the following:

(i) for each $i\in\{1,\dots,l_n\}$, $x^n({i})\in A_{a(i)}$ and

(ii) for each $i\in\{1,\dots,l_n-1\}$, $d(x^n(i),x^n(i+1))\leq d_H(A_{a^n(i)},A_{a^n(i+1)})$.

We may assume that $(C_n)_n$ converges to $C$. By property (ii) and the uniform continuity of $h$, $C$ is a sub-continuum of $X$ (see Lemma 4.16 in \cite{Nadler}).

Notice that if $x\in A_t\cap C$  and $\lim_{n\to +\infty}x^n(k_n)=x$ where $t\in[0,1]$ and $0\leq k_n\leq l_n$ for all $n\in\mathbb{N}$ then $\lim_{n\to +\infty}a^n(k_n)=t$. We claim that $\{f^n(C):n\in\mathbb{Z}\}$ is a family of pairwise disjoint sets. Otherwise, for some $n\in \mathbb{N}$ there exists $z\in C$ such that $f^n(z)\in C$. Let $t\in[0,1]$ be such that $z,f^n(z)\in A_t$. Hence there are two sequences $(k_n)_n$ and $(s_n)_n$ such that $\lim_{n\to +\infty}x^n(k_n)=z$ and $\lim_{n\to +\infty}x^n(s_n)=f^n(z)$ where for each $n\in\mathbb{N}$, $k_n,s_n\in \{0,\dots,l_n\}$. Thus $\lim_{n\to +\infty}a^n(s_n)=\lim_{n\to +\infty}a^n(k_n)=t$. Without loss of generality we may assume that $s_n<k_n$ for every $n$. It turns out that $(\{x^n_i: s_n\leq i\leq k_n\})_n$ have at least one convergent subsequence with limit a non degenerate sub-continuum of $C$ included into $A_t$ and containing both $z$ and $f^n(z)$. Therefore the connected component of $A_t$ that contains $z$ is $f^n$-invariant since it contains also $f^n(z)$ so $A_t\in\mathcal{M}_c(f)$, a contradiction.


By property (i), $C$ contains one point from each $A_0$ and $A_1$ so it is non-degenerate. It follows that $\{f^n(C):n\in\mathbb{N}\}$ is family of pairwise disjoint connected subsets of $X$ each one from this family contains one point from $A_0$ and another from $A_1$ so each one from this family has diameter more at least $d(A_0,A_1)=\inf\{d(x,y): x\in A_0$ and $y\in A_1\}>0$. This is impossible since $X$ is in particular finitely Suslinian (see Corollary 1.7 in \cite{lel}). We conclude then that $\mathcal{M}(f)$ is a singleton, that is there a unique minimal set for $f$.
\qed

\medskip

\begin{cor}
Assume that $f$ is a homeomorphism of a regular curve having at least two minimal sets then it has a periodic point.
\end{cor}

\medskip

\begin{cor}
Assume that $f$ is a homeomorphism of a regular curve without periodic point and $M$ its unique minimal set then $f_{|M}$ is totally minimal.
\end{cor}

\textit{Acknowledgements.} This work was supported by the research unit:
``Dynamical systems and their applications'' (UR17ES21), Ministry of Higher Education and Scientific Research,
Tunisia.

\bibliographystyle{amsplain}

\end{document}